\documentclass[preprint,12pt]{elsarticle}
\usepackage{amsfonts,amssymb, amsmath, amsthm}
\usepackage{fullpage}
\usepackage{graphicx}
\usepackage{float}
\restylefloat{figure}

\journal{Applied Mathematics and Computation}

\numberwithin{equation}{section}

\newtheorem{guess}{Theorem}
\newtheorem{uess}{Lemma}
\newtheorem{corollary}{Corollary}
\newtheorem{example}{Example}
\newtheorem{remark}{Remark}

\begin{document}

\begin{frontmatter}

\title{Stability conditions for scalar delay differential equations with a non-delay term} 

\author[label1]{Leonid Berezansky} 
\author[label2]{Elena Braverman}
\address[label1]{Department of Mathematics,
Ben-Gurion University of the Negev,
Beer-Sheva 84105, Israel}
%
%and  Elena Braverman $^{2}$ \\
\address[label2]{Department of Mathematics and Statistics, University of
Calgary,
2500 University Drive N.W., Calgary, AB T2N 1N4, Canada}

%\date{}
%\maketitle

\begin{abstract}
The problem considered in the paper is exponential stability of linear equations
and global attractivity of nonlinear non-autonomous equations which include a non-delay term and one or more
delayed terms. 
First, we demonstrate that introducing a non-delay term with a non-negative coefficient 
can destroy stability of the delay equation. 
Next, sufficient exponential stability conditions 
for linear equations with concentrated or distributed delays and global attractivity conditions 
for nonlinear equations are obtained. The nonlinear results are applied to 
the Mackey-Glass model of  respiratory dynamics. 
\end{abstract}


\end{frontmatter}

\section{Introduction}

Stability of the autonomous delay differential equation
\begin{equation}\label{intro1}
\dot{x}(t)+b x(t-\tau)=0
\end{equation}
(the sharp asymptotic stability condition for $\tau>0$ is $0<b\tau < \pi/2$) and  
of the equation with a non-delay term
\begin{equation}\label{intro2}
\dot{x}(t)+ax(t)+b x(t-\tau)=0
\end{equation} 
was investigated in detail, and stability of (\ref{intro1}) implies stability of (\ref{intro2}) for any $a \geq 0$.

The equation
\begin{equation}\label{1}
\dot{x}(t)+ax(t)+b(t)x(h(t))=0,~ t\geq 0,
\end{equation}
where $a>0$ is a constant, $b$ is a locally essentially bounded nonnegative function, $h(t) \leq t$ is a delay 
function,
is a generalization of (\ref{intro2}) and also is a special case of the non-autonomous equation with two 
variable coefficients
\begin{equation}\label{intro3}
\dot{x}(t)+a(t)x(t)+b(t)x(h(t))=0,~ t\geq 0,~a(t) \geq 0.
\end{equation}
Let us note that, generally, asymptotic stability of the equation without the non-delay term
\begin{equation}\label{intro4}
\dot{x}(t)+b(t)x(h(t))=0,~ t\geq 0
\end{equation}
does not imply stability of (\ref{intro3}).

\begin{example}
\label{example1}
Consider equations (\ref{intro3}) and (\ref{intro4}) for $b(t)\equiv b>0$ and $h(t)=[t]$, where $[t]$ is the 
maximal integer not exceeding $t$. The equation 
\begin{equation}\label{intro5}
\dot{x}(t)+b x([t])=0, ~t \geq 0
\end{equation} 
is asymptotically stable for any $b$ satisfying $0<b<2$, since the solution on $[n,n+1]$ is 
$x(t)=x(n)[1-b(t-n)]$ which is a linear function on any $[n,n+1]$.
Thus $x(n)=(1-b)^nx(0)$ and $|x(n)| \leq \delta^n |x(0)|$, where $0<\delta = |1-b| <1$.

Let us choose $1.6<b<1.9$ and consider the equation

\begin{equation}\label{intro6}
\dot{x}(t)+a(t)x(t)+b x([t])=0, ~t \geq 0
\end{equation} 
with a periodic piecewise constant nonnegative function $a(t)$ with the period $T=1$.
If $a(t) \equiv \alpha$ on $[0, \varepsilon]$ for $0<\varepsilon<1$ then
$$ x(t) = \left( \frac{b}{\alpha} +1 \right)x(0) e^{-\alpha t} - \frac{b}{\alpha} x(0), ~t \in [0, \varepsilon].$$ 
Let us choose $\alpha=3b$ and $\varepsilon$ in such a way that $x(\varepsilon)=0$, i.e. $\varepsilon = 
\frac{1}{3b} \ln 
4$,
and
\begin{equation}
\label{a_def}
a(t)= \left\{ \begin{array}{ll} 3b, & n \leq t \leq n+\varepsilon, \\ 
0, & n+\varepsilon<t<n+1, \end{array} \right. 
\end{equation}
where $n \geq 0$ is an integer. 
For $1.6<b<1.9$ we have $0.24< \varepsilon < 0.29$, thus $|x(1)| = b|x(0)|(1- \varepsilon) >  1.136 |x(0)|$.
Further, $|x(n)| > 1.136^n |x(0)|$, which means that (\ref{intro6}) is unstable, 
while (\ref{intro5}) is asymptotically stable.
Fig.~\ref{figure1}, left, illustrates the solutions of (\ref{intro5}) and (\ref{intro6}) 
with $b=1.8$, $x(0)=1$, here $|x(n+1)| \approx 1.34 |x(n)|$ for (\ref{intro6}), so (\ref{intro6}) is unstable
while (\ref{intro5}) is stable.

It is also possible to construct an example of asymptotically stable equation (\protect{\ref{intro5}}) with $a(t)$ satisfying 
$\inf_{t>1} a(t)>0$ such that (\protect{\ref{intro6}}) is unstable. For example, consider 
\begin{equation}
\label{a_def1}
a(t)= \left\{ \begin{array}{ll} 3b, & n \leq t \leq n+\varepsilon, \\
0.5, & n+\varepsilon<t<n+1, \end{array} \right.     
\end{equation}
where $b=1.8$, $x(0)=1$. As previously, $x(t)=\frac{4}{3}x(n) e^{-\alpha (t-n)} - 
\frac{1}{3} x(n)$ on $[n,n+\varepsilon]$; the solution on $[n+\varepsilon,n+1]$ is
$x(t)=2bx(n)(e^{-0.5(t-n-\varepsilon)} -1)$ and $|x(n+1)| \approx 
1.12|x(n)|$ for (\ref{intro6}).
In this case $a(t) \geq 0.5$ for any $t$, and the solution is unstable and unbounded (see 
Fig.~\ref{figure1}, right), though the divergence is slower than in the case when $a$ is defined by 
(\ref{a_def}).

\begin{figure}[ht]
\centering
\includegraphics[height=.28\textheight]{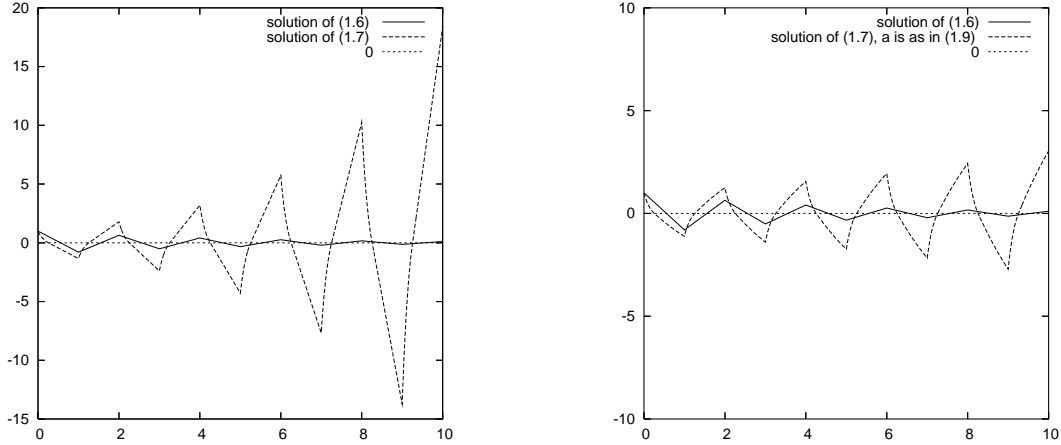} ~~~~~~~~~~
\includegraphics[height=.28\textheight]{example_1_a05.ps}
\caption{Solutions of equations (\protect{\ref{intro5}}) and (\protect{\ref{intro6}})
with $b=1.8$, $x(0)=1$, $\varepsilon \approx 0.256721$ in the case when $a$ is defined
by (\protect{\ref{a_def}}) and can vanish (left) and $a$ is described by (\protect{\ref{a_def1}}) 
and satisfies $a(t) \geq 0.5$ (right). All the solutions are oscillatory, 
(\protect{\ref{intro5}}) is exponentially stable, while (\protect{\ref{intro6}})
is unstable in both cases.}
\label{figure1}
\end{figure}

\end{example}

For scalar differential equation (\ref{1}), 
where $a>0$ is a constant, $b$ is a locally essentially bounded nonnegative function, $h(t) 
\leq t$ is a delay function,
the following result is a corollary of \cite[Theorem 2.9]{ILT}.

\begin{guess}\label{th1}
Suppose $0\leq b(t)\leq b$, $0\leq t-h(t)\leq h$ and the inequality 
\begin{equation}\label{2}
\frac{a}{b}e^{-ah}>\ln\frac{b^2+ab}{b^2+a^2}
\end{equation}
holds. Then equation (\ref{1}) is exponentially stable.
\end{guess}  

%\end{document}
The aim of this paper is to extend Theorem~\ref{th1} to other classes of equations, 
including \eqref{intro3}, models with variable coefficients and several delays,
as well as with distributed delays. In Section 3 we consider nonlinear delay differential equations
and apply the results obtained to the Mackey-Glass model of respiratory dynamics in Section 4.

For other recent stability results, different from the results in the present paper,  for linear scalar delay differential 
equations see \cite{YS1,Y,Kz,SYC,GH1,GH2,GD,AS,Mal2,GP,BB,Pham,BB2007}
and in \cite{M, LP, GT, Gil1, KM, K, FO} for nonlinear equations.

\section{Linear Equations}

Consider the equation
\begin{equation}\label{3}
\dot{x}(t)+a(t)x(t)+b(t)x(h(t))=0, t\geq 0,
\end{equation}
under the following assumptions: 

(a1) $a,b$ are essentially bounded on $[0,\infty)$ Lebesgue measurable nonnegative functions;

(a2) $h$ is a Lebesgue measurable function, $h(t)\leq t, \lim_{t\rightarrow\infty}h(t) =\infty$. 

Together with (\ref{3}) consider the initial condition
\begin{equation}\label{4}
x(t)=\varphi(t),~ t\leq 0.
\end{equation}
We assume that

(a3) $\varphi$ is a Borel measurable bounded function.

The solution of problem~(\ref{3})-(\ref{4}) is an absolutely continuous on $[0,\infty)$ function satisfying (\ref{3}) almost 
everywhere for 
$t\geq 0$ and condition (\ref{4}) for $t \leq 0.$
Instead of the initial point $t_0=0$ we can consider any $t_0>0$.
\begin{guess}
\label{th2}
Suppose $a(t)\geq a_0>0$,  $b(t)\geq 0$,
\begin{equation}\label{5}
h_0 :=\limsup_{t\rightarrow\infty} \int_{h(t)}^t a(s) ds<\infty,
\end {equation}
and the inequality 
\begin{equation}\label{6}
\frac{1}{\beta}e^{-h_0}>\ln\frac{\beta^2+\beta}{\beta^2+1}
\end{equation}
holds, where
\begin{equation}\label{6a}
\beta := \limsup_{t\rightarrow\infty}\frac{b(t)}{a(t)} \,.
\end{equation}
Then equation (\ref{3}) is exponentially stable.
\end{guess}  
\begin{proof} 
By (\ref{6}), with the notation introduced in (\ref{5}) and (\ref{6a}), there exists $t_0\geq 0$  such that the inequality
$$
\frac{1}{\beta}e^{-H}>\ln\frac{B^2+B}{B^2+1},
$$
holds, where
$$
H=\sup_{t\geq t_0} \int_{h(t)}^t a(s) ds,~
B = \sup_{t\geq t_0}\frac{b(t)}{a(t)}.
$$
Without loss of generality we can assume $t_0=0$.
After the substitution
$$
s=p(t)=\int_0^t a(\tau)d\tau,~ y(s)=x(t)
$$
(the function $p(t)$ is one-to-one since $a(t) \geq a_0>0$), 
equation (\ref{3}) has the form
\begin{equation}\label{7}
y'(s)+y(s)+\frac{b(p^{-1}(s))}{a(p^{-1}(s))}\, y(l(s))=0,
\end{equation}
where $l(s)=\int_0^{h(p^{-1}(s))} a(\tau)d\tau$. 
Moreover, the function $p(t)$ is monotone increasing and absolutely continuous,
therefore $p^{-1}(t)$ is also a continuous increasing function. Thus 
$h(p^{-1}(\cdot))$, $a(p^{-1}(\cdot))>0$ and $b(p^{-1}(\cdot))$
are Lebesgue measurable functions as compositions of a continuous and a Lebesgue measurable function.
Therefore the coefficients and the arguments in equation (\ref{7})
are Lebesgue measurable.
We have
$$
\frac{b(p^{-1}(s))}{a(p^{-1}(s))}=\frac{b(t)}{a(t)}\leq B,~~
s-l(s)=\int_{h(p^{-1}(s))}^{p^{-1}(s)} a(\tau)d\tau \leq H, ~~ s \geq p(t_0).
$$
By Theorem~\ref{th1}  equation~(\ref{7}) is exponentially stable. 
It means that there exist $M>0$ and $\alpha>0$ such that for any solution $y$
of equation~(\ref{7}) with the initial function $\varphi$  the inequality
$
|y(s)|\leq M\|\varphi\|e^{-\alpha s}
$ holds, where $\| \cdot \|$ is the sup-norm.
Thus for the solution $x(t)=y(s)$ of problem~(\ref{3}),(\ref{4}) we have
$$
|x(t)|\leq M\|\varphi\|e^{-\alpha \int_0^t a(\tau)d\tau}\leq M\|\varphi\|e^{-\alpha a_0 t}.
$$
Hence equation~(\ref{3}) is exponentially stable, which concludes the proof.
\end{proof}

Consider the equation with several delays
\begin{equation}\label{9}
\dot{x}(t)+a(t)x(t)+\sum_{k=1}^m b_k(t)x(h_k(t))=0,~~ t\geq 0,
\end{equation}
where for the functions $a, b_k, h_k $ conditions (a1)-(a2) hold.

\begin{guess}\label{th3}
Suppose $a(t)\geq a_0>0, b_k(t)\geq 0$, 
\begin{equation}\label{10}
h_0 :=\limsup_{t\rightarrow\infty} \int_{\min_k h_k(t)}^t a(s)ds<\infty,
\end {equation}
and inequality ~(\ref{6}) holds,
where $b(t)=\sum_{k=1}^m b_k(t)$, $\beta$ is defined in (\ref{6a}).

Then equation (\ref{9}) is exponentially stable.
\end{guess}  
\begin{proof}
Suppose $x$ is a solution of equation (\ref{9}).
The functions defined as
\begin{equation}\label{10abc}
\underline{h}(t):=\min_{1 \leq k \leq m} h_k(t), ~~~
u(t):=\left. \sum_{k=1}^m b_k(t)x(h_k(t))\right/ \sum_{k=1}^m b_k(t)
\end {equation}
are both Lebesgue measurable.  
Define
\begin{equation}\label{10ab}
h(t)=\inf_{s \in [\underline{h}(t),t]} \{ s| x(s) =u(t) \},
\end {equation}
the fact that the set $ \{ s\in [\underline{h}(t),t]| x(s) =u(t) \}$  is non-empty was justified
in  \cite[Lemma 5]{BB1}.
Further, let us notice that for any $C>0$ and $u$, $h$ defined in (\ref{10abc}) and (\ref{10ab}), respectively,
the set $\{t | h(t) \leq C\}$ has the form
$$
\{t | h(t) \leq C\}=\left\{ t \left| \max_{s \in [\underline{h}(t),C]} x(s) \geq u(t)~~\mbox{or~~~} t \leq C \right. \right\}
= \left\{ t \left| \max_{s \in [\underline{h}(t),C]} x(s) \geq u(t) \right. \right\} \cup [0,C].
$$
Since $x:[0,\infty) \to {\mathbb R}$ is continuous and $\underline{h}(t)$ is measurable, the function
$\displaystyle \max_{s \in [\underline{h}(t),C]} x(s)$ is a Lebesgue measurable function of $t$. Therefore
the set $ \left\{ t \left| \max_{s \in [\underline{h}(t),C]} x(s) \geq u(t) \right. \right\}$
is measurable for any $C$, which by definition implies that $h$ is measurable.
Since $u(t)=x(h(t))$ then $x$ is a solution of  equation~(\ref{3}) with nonnegative measurable coefficients and a 
measurable delay which is  exponentially stable by Theorem~\ref{th2}.
Thus equation (\ref{9}) is also exponentially stable.
\end{proof}

Consider now the equation with a distributed delay
\begin{equation}\label{11}
\dot{x}(t)+a(t)x(t)+\sum_{k=1}^m b_k(t)\int_{h_k(t)}^t x(s)d_s R_k(t,s)=0,
\end{equation}
where for $a,b_k,h_k$ conditions (a1)-(a2) hold, $\varphi$ in (\ref{4}) is continuous and

(a4) $R_k(t,s)$ are nondecreasing in $s$ for almost all $t$ and $\int_{0}^t d_s R_k(t,s)\equiv 1$, $k=1, \dots, 
m$.

\begin{guess}\label{th4}
Suppose $a(t)\geq a_0>0, b_k(t)\geq 0$, conditions (\ref{10}) and (\ref{6}) hold, 
where $b(t)=\sum_{k=1}^m b_k(t)$, $\beta$ is defined in (\ref{6a}).
Then equation (\ref{11}) is exponentially stable.
\end{guess}  
\begin{proof}
Suppose $x$ is a solution of equation (\ref{11}).
By \cite[Theorem 9]{BB2}, there exists a function $g(t)\leq t$ such that 
$\min_{1 \leq k \leq m} h_k(t)\leq g(t)\leq t$ 
and any solution of (\ref{11}) is also a solution of the equation
\begin{equation}\label{12}
\dot{y}(t)+a(t)y(t)+ \left( \sum_{k=1}^m b_k(t) \right) y(g(t))=0.
\end{equation}
The fact that $g(t)$ can be chosen as a Lebesgue measurable function is verified similarly to
the proof of Theorem~\ref{th3}.
By Theorem~\ref{th2} equation (\ref{12}) and thus equation (\ref{11}) are exponentially stable.
\end{proof}

Consider now the integro-differential equation
\begin{equation}\label{13}
\dot{x}(t)+a(t)x(t)+\sum_{l=1}^m b_l(t)\int_{h_l(t)}^t K_l(t,s) x(s)ds =0,
\end{equation}
where for $a,b_l,h_l$ conditions (a1)-(a2) hold and 

(a5) $K_l(t,s)\geq 0$ are essentially bounded and $\int_{h_l(t)}^t K_l(t,s)ds \equiv 1$, $l=1, \dots, m$.

\begin{corollary}
\label{cor1}
Suppose $a(t)\geq a_0>0, b_l(t)\geq 0$, conditions (\ref{10}) and (\ref{6}) hold, 
where $b(t)=\sum_{l=1}^m b_l(t)$, $\beta$ is defined in (\ref{6a}).
Then equation (\ref{13}) is exponentially stable.
\end{corollary}

\section{Nonlinear Equations}

Consider now the nonlinear equation
\begin{equation}\label{14}
\dot{x}(t)+f(t,x(t))+\sum_{k=1}^m g_k(t, x(h_k(t)))=0
\end{equation}
with initial condition  (\ref{4}), where everywhere in this section we assume that the functions $h_k$,
$k=1, \dots, m$, satisfy (a2), (a3) and the following conditions hold:

(a6) $f(t,u)$, $g_k(t,u)$ are continuous, $f(t,0)=g_k(t,0)=0$, $f(t,u)u>0$, $g_k(t,u)u>0$ for any $u\neq 0$ and
$k=1, \dots, m$;

(a7) there exist $x_1^0,x_2^0,x_1,x_2$, where $-\infty\leq x_1^0\leq 0\leq x_2^0\leq \infty$ and $-\infty< x_1\leq 0\leq 
x_2< \infty$ such that 
for any $x_1^0\leq \varphi\leq x_2^0$ there exists the unique global solution $x$ of problem (\ref{14}), (\ref{4}), and
it satisfies $x_1\leq x(t)\leq x_2$.

\begin{guess}\label{th5}
Suppose that there exist positive numbers $a_0$,$A$,$b_k$, $k=1, \dots, m$ such that for any $x_1\leq u\leq x_2, 
u\neq 0$ we have
$$
a_0\leq \frac{f(t,u)}{u}\leq A,~ 0\leq \frac{g_k(t,u)}{u}\leq b_k.
$$ 
Assume also that $t-h_k(t)\leq h_0, b_0=\sum_{k=1}^m b_k$ and
\begin{equation}\label{15}
\frac{a_0}{b_0}e^{-A h_0}>\ln\frac{b_0^2+a_0 b_0}{b_0^2+a_0^2}.
\end{equation}
Then all solutions of problem (\ref{14}),  (\ref{4}) with  $x_1^0\leq \varphi\leq x_2^0$ converge to zero.
\end{guess}
\begin{proof}
Suppose $x$ is a solution of problem (\ref{14}),  (\ref{4}) with  $x_1^0\leq \varphi\leq x_2^0$.
Denote
$$
a(t)=\left\{\begin{array}{ll}
\frac{f(t,x(t))}{x(t)},& x(t)\neq 0,\\
0, & x(t)=0,
\end{array}\right.
$$$$
b_k(t)=\left\{\begin{array}{ll}
\frac{g_k(t,x(h_k(t)))}{x(h_k(t))},& x(h_k(t))\neq 0,\\
0, & x(h_k(t))=0,
\end{array}\right.
$$
then equation (\ref{14}) has form (\ref{9}). 
All conditions of Theorem~\ref{th3} are satisfied with $\displaystyle \beta=\frac{b_0}{a_0}$ and $Ah_0$ instead of $h_0$
in (\ref{10}), 
hence for any solution $y$ of equation  (\ref{9})
we have $\lim_{t\rightarrow\infty} y(t)=0$. Then $\lim_{t\rightarrow\infty} x(t)=0$.
\end{proof}

Consider now the nonlinear equation with a distributed delay
\begin{equation}\label{16}
\dot{x}(t)+f(t,x(t))+\sum_{k=1}^m  \int_{h_k(t)}^t g_k(t,x(s)) d_s R_k(t,s)=0,
\end{equation}
where conditions (a2),(a4),(a6) and (a7) hold, the initial function $\varphi$ is continuous.

\begin{guess}\label{th6}
Assume that for any $x_1\leq u\leq x_2, u\neq 0$ 
$$
a_0\leq \frac{f(t,u)}{u}\leq A,~ 0\leq \frac{g_k(t,u)}{u}\leq b_k.
$$ 
Assume also that $t-h_k(t)\leq h_0, b_0=\sum_{k=1}^m b_k$ and inequality (\ref{15}) holds.
Then the zero solution is an attractor of all solutions of problem (\ref{16}), (\ref{4}) with
the initial function satisfying  $x_1^0\leq \varphi\leq x_2^0$.
\end{guess}
The proof applies Theorem~\ref{th4} and is similar to the proof of Theorem~\ref{th5}.

\begin{remark} 
Nonlinear integro-differential equations, mixed differential equations
with concentrated delay and integral terms are partial cases of equation (\ref{16}).
\end{remark}

\section{Mackey-Glass Model of Respiratory Dynamics}

As an  application we consider the Mackey-Glass model of respiratory dynamics (for review and recent results see \cite{BBI})
\begin{equation}\label{17}
\dot{x}(t)= r(t)\left[\alpha-\frac{\beta x(t)x^{n}(h(t))}{1 +x^{n}(h(t))}\right],
\end{equation}
where $\alpha>0, \beta>0$ and $n>0$ are positive constants, $R\geq r(t)\geq r_0>0$ is a Lebesgue measurable function,
$h(t)\leq t$ is a measurable delay function, $t-h(t)\leq h_0$.
Equation (\ref{17}) has a nontrivial equilibrium $K$, where $K$ is a unique positive solution
determined by the equation 
\begin{equation}\label{18}
\beta K^{n+1} =\alpha (1+K^{n}).
\end{equation}

\begin{uess}\label{lemma1}\cite[Lemma 3.1]{BBI}
For any $\varphi(t)\geq 0$,  $\varphi(0)>0$, problem (\ref{17}), (\ref{4}) has a unique global positive 
solution. 

For any $\varepsilon>0$ there exists sufficiently large $t_1$ such that for $t\geq t_1$ the solution satisfies 
$\mu_{\varepsilon}\leq x(t)\leq M_{\varepsilon}$, where
 \begin{equation}
\label{19}
\mu_{\varepsilon}=\frac{\alpha}{\beta}-\varepsilon, 
~~M_{\varepsilon}=\frac{\alpha}{\beta}\left[1+\left(\frac{\beta}{\alpha}\right)^n\right]+\varepsilon.
\end{equation}
\end{uess}

After the substitution $y(t)=\ln\frac{x(t)}{K}$ equation (\ref{17}) has the form
\begin{equation}\label{20}
\dot{y}(t)+r(t)\frac{\alpha}{K}\left( 1-e^{-y(t)} \right) +\beta K^n  r(t)
\left(\frac{1}{1+K^ne^{-n y(h(t))}}-\frac{1}{1+ K^n}\right)=0.
\end{equation}

\begin{uess}\cite[Theorem 3.3]{BBI}\label{lemma2} 
For any $\varepsilon>0$ and sufficiently large $t$, for any solution $y$ of problem (\ref{20}),(\ref{4}),
the inequality   
$c_{\varepsilon}\leq y(t)\leq C_{\varepsilon}$ is satisfied, where
\begin{equation}
\label{21}
c_{\varepsilon}= \ln\frac{\mu_{\varepsilon}}{K}, C_{\varepsilon}=\ln\frac{M_{\varepsilon}}{K},
\end{equation}
and $\mu_{\varepsilon}, M_{\varepsilon}$ are denoted by (\ref{19}).
 \end{uess}

Denote 
$$
 \mu=\frac{\alpha}{\beta}, ~~M=\frac{\alpha}{\beta}\left[1+\left(\frac{\beta}{\alpha}\right)^n\right],
~~ c= \ln\frac{\mu}{K}, ~~ C=\ln\frac{M}{K}.
$$

\begin{guess}\label{th7}
Suppose $t-h(t)\leq h_0$, $0<r_0 \leq r(t)<R$ and inequality~(\ref{15}) holds, where
$$
a_0=\frac{\alpha}{K}\frac{1-e^{-C}}{C}r_0,~~ A=\frac{\alpha}{K}\frac{1-e^{-c}}{c}R,~~ 
b_0=\frac{\beta n R}{4}.
$$
Then $K$ is a global attractor for all solutions of problem (\ref{17}), (\ref{4})
with $\varphi(t) \geq 0$, $\varphi(0)>0$.
\end{guess} 
\begin{proof}
It is sufficient to prove that $y(t)=0$ is a global attractor for all solutions of  problem (\ref{20}),(\ref{4}).
By Lemma~\ref{lemma2}, there exist $\varepsilon>0$ and $t_1\geq 0$ such that the solution of problem 
(\ref{20}),(\ref{4})
satisfies  $c_{\varepsilon}\leq y(t)\leq C_{\varepsilon}$ for $t\geq t_1$, and 
inequality (\ref{15}) holds if $a_0,A$
are changed by 
$$
a_{\varepsilon}=\frac{\alpha}{K}\frac{1-e^{-C_{\varepsilon}}}{C_{\varepsilon}}r_0, ~~
A_{\varepsilon}=\frac{\alpha}{K}\frac{1-e^{-c_{\varepsilon}}}{c_{\varepsilon}}, 
$$
respectively, where $c_{\varepsilon}, C_{\varepsilon}$ are denoted by (\ref{21}).

Equation (\ref{20}) has form (\ref{14}) for $m=1$ with
$$
f(t,x)=r(t)\frac{\alpha}{K}(1-e^{-x}),~~g(t,x)=\beta K^n  r(t)
\left(\frac{1}{1+K^n e^{-n x}}-\frac{1}{1+ K^n}\right).
$$
In \cite[the proof of Theorem 5.4]{BBI} for these functions the following inequalities were justified:
$$
a_{\varepsilon}\leq \frac{f(t,u)}{u}\leq A_{\varepsilon},~ 0\leq \frac{g(t,u)}{u}\leq b_0.
$$
By Theorem~\ref{th6}, the zero solution is a global attractor for all solutions of  problem (\ref{20}), 
(\ref{4}).
\end{proof}

\begin{example}
\label{example2} 
Consider equation (\ref{17}) with $K=1.5$, $\alpha=1$, $\beta = 0.5$,
$n=4$, $r(t) = 2.7+0.3 \sin t$, $t-h(t) \leq h_0$.
   
To apply Theorem~\ref{th7}, we compute $R=3$, $r_0=2.4$,
$\mu=2$, $M =2.125$, $c \approx 0.28768$, $C \approx 0.34831$,
$a_0\approx 1.35107$, $b_0 = 1.5$, $A \approx 1.73803$
and obtain that $K$ is a global attractor if $h_0<1.68$. 

For comparison, \cite[Theorem 5.4]{BBI} gives the condition
${\displaystyle 
\frac{\beta h_0 n R}{4}<1+\frac{1}{e}} $
for the global attractivity of $K$ which leads to the estimate $h_0<0.91$.
The results of \cite[Corollary 4]{ILT2} cannot be applied since the coefficients are variable.
\end{example}

\section{Discussion}

%\begin{problem}
Everywhere above for linear equations 
\begin{equation}\label{1_disc}
\dot{x}(t)+a(t)x(t)+b(t)x(h(t))=0
\end{equation}
we assumed a positive lower bound $a(t) \geq a_0>0$ for the 
coefficient of the non-delay term.
Moreover, if the results of the present paper imply stability for a 
certain bound $a_0$, they would also yield that the equation is stable for any greater lower bound.
However, Example~\ref{example1} demonstrated that in a stable equation with a single delay 
term (\ref{intro4}) which has a positive variable coefficient,  the introduction of a 
non-delay term with a nonnegative (or even positive)
coefficient as in (\ref{intro3}) may destroy its stability.

Let us note that the condition
\begin{equation}
\label{2_disc}
\int_{h(t)}^t b(s)~ds <\frac{1}{e}
\end{equation}
guarantees that (\ref{intro4}) is stable and also that (\ref{1_disc}) is stable for any $a(t)\geq 0$ as 
(\ref{2_disc}) implies nonoscillation (and thus stability) of (\ref{intro4}).
In fact, denoting $\displaystyle z(t) = x(t) \exp\left\{ \int_0^t a(s)~ds \right\}$, we can rewrite (\ref{1_disc}) 
as 
\begin{equation}
\label{3_disc}
\dot{x}(t)+ r(t)z(h(t))=0, ~~r(t)=b(t)e^{-\int_{h(t)}^t a(s)~ds},
\end{equation}
where nonoscillation of $z$ is equivalent to nonoscillation of $x$.
For any $a(t) \geq 0$, equation (\ref{3_disc}) is nonoscillatory as (\ref{2_disc}) implies
\begin{equation}
\label{4_disc}
\int_{h(t)}^t r(s)~ds <\frac{1}{e},
\end{equation}
thus (\ref{1_disc}) is stable (and even nonoscillatory). The possibility to destabilize
oscillatory solutions was illustrated in Example~\ref{example1}. However,
it is still an open problem whether some other conditions
which would guarantee that stability of (\ref{intro4}) implies stability of (\ref{intro3}) can be established,
where (\ref{2_disc}) does not hold, and the inequality $0<b(t)<\lambda a(t)$ is not satisfied for any $0<\lambda<1$
(the latter inequality would imply stability \cite{HaleLunel}).

%\end{problem}

%\begin{problem}
In the present paper, global attractivity of the trivial equilibrium for 
nonlinear equations of form  (\ref{14}) was
considered, where 
$f(t,0)=g_k(t,0)=0$, $f(t,u)u>0$, and $g_k(t,u)u>0$ for $u\neq 0$. 
Such equations are obtained from a given mathematical model after the substitution $x=y+K$,
where $K$ is a positive equilibrium or a positive periodic/almost periodic solution.

However, in (\ref{14}) every term in the sum contains only one delay. 
It would be interesting to extend global stability results obtained 
here to more general equations, for example, of the form
$$
\dot{x}(t)+f(t,x(t))+\sum_{k=1}^m g_k(t,  x(h_1(t)),\dots, x(h_l(t)))=0.
$$
%\end{problem}
\vspace{2mm}

\centerline{\bf Acknowledgment}

The first author was partially supported by Israeli Ministry of
Absorption. The second author was
partially supported by the NSERC Discovery Grant.

\end{document}